\documentclass[11pt,twoside]{amsart}
\usepackage{amsmath, amsthm, amscd, amsfonts, amssymb, graphicx, color}
\usepackage[bookmarksnumbered, colorlinks, plainpages]{hyperref}

\newtheorem{thm}{Theorem}[section]
\newtheorem{cor}[thm]{Corollary}
\newtheorem{lem}[thm]{Lemma}

\newtheorem{cl}{Claim}
\numberwithin{equation}{section}

\newcommand{\U}{\mathcal{U}}
\newcommand{\X}{\mathcal{X}}
\newcommand{\Ll}{(\textbf{L}_{w})}
\begin{document}


\vspace{1.3 cm}

\title[Characterization of Lie centralizable mappings on $\mathcal{B}(\mathcal{X})$]{Characterization of Lie centralizable mappings on $\mathcal{B}(\mathcal{X})$}

\author{Behrooz Fadaee} 
\address{Department of Mathematics, Faculty of Science, University of Kurdistan, P.O. Box 416, Sanandaj, Kurdistan, Iran.}
\email{behroozfadaee@yahoo.com; b.fadaee@uok.ac.ir}

\author[Hoger Ghahramani]{ Hoger Ghahramani$^{\ast}$}
\address{Department of Mathematics, Faculty of Science, University of Kurdistan, P.O. Box 416, Sanandaj, Kurdistan, Iran.}
\email{hoger.ghahramani@yahoo.com; h.ghahramani@uok.ac.ir}

\author{Ayyoub Majidi} 
\address{Department of Mathematics, Faculty of Science, University of Kurdistan, P.O. Box 416, Sanandaj, Kurdistan, Iran.}
\email{ayyoubmajidi@gmail.com; a.majidi@uok.ac.ir}

\thanks{{\scriptsize
\hskip -0.4 true cm MSC(2020): Primary: 47B47, 47L10.
\newline Keywords: Lie centralizer, Algebras of operators on Banach spaces.
\newline $^{\ast}$ Corresponding author.
\\}}

\maketitle

\begin{center}
\end{center}

\begin{abstract}  
Let $\mathcal{B}(\mathcal{X})$ be the algebra of all bounded linear operators on a complex Banach space $\mathcal{X}$, and let $W\in \mathcal{B}(\mathcal{X}) $ is such that $\overline{W(\mathcal{X})}\neq \mathcal{X}$ or $W=\xi I$ where $\xi\in \mathbb{C}$ and $I$ be the identity operator. We show that if $\phi:\mathcal{B}(\mathcal{X}) \rightarrow \mathcal{B}(\mathcal{X})$ is an additive mapping Lie centralizable at $W$ (i.e. $\phi([A,B])=[\phi(A),B]=[A,\phi(B)]$ for any $A,B\in \mathcal{B}(\X) $ with $AB=W$), then $ \phi (A) = \lambda A + \mu (A) $ for all $A\in \mathcal{B}(\mathcal{X})$, where $ \lambda \in \mathbb{C} $ and $\mu: \mathcal{B}(\mathcal{X})\rightarrow \mathbb{C}I$ is an additive mapping such that $ \mu ([A,B]) = 0$ for all $A,B \in \mathcal{B}(\mathcal{X})$ with $AB=W$.
\end{abstract}

\vskip 0.2 true cm



\section{Introduction}
Let $\U$ be an algebra. The additive mapping $\phi:\U \rightarrow \U$ is said to be a \textit{Lie centralizer} if $\phi([a,b])=[\phi(a),b]$ for each $a,b\in \U$, where $[a,b]=ab-ba$ is the Lie product of $a$ and $b$ in $\U$. A routine check shows that $\phi$ is a Lie centralizer on $\U$ if and only if $\phi([a,b])=[a,\phi(b)]$ for all $a,b\in \U$. Lie centralizers are an important class of mappings related to the Lie structure of algebras. This is a classical notion in the theory of Lie and other non-associative algebras. We are only interested in the case, where associative algebras are endowed with the Lie product and create a Lie algebra. Lie centralizers have recently been studied a lot on algebras from different perspectives (For instance , see \cite{beh, fadghah3, fadghah4, fo, gh1, ghJordan, gh0, ghi0, ghi, jab, liu2} and the references therein). One of these study paths has been the characterization of Lie centralizers at specific products. To be precise, this study path consists of characterizing the mapping $\phi:\U \rightarrow \U$ that fulfills the following condition:
\[a,b\in\U, \,  ab=w\Longrightarrow \phi([a,b])=[\phi(a),b]=[a,\phi(b)] \quad \Ll,\]
where $w\in \U$ is fixed. The mapping $\phi$ satisfying $\Ll$ is called \textit{Lie centralizable at $w$}. Determined mappings with local properties have always been considered in mathematics. Among the local properties, we can mention mappings on rings or algebras that act on special products such as Lie centralizers. We refer to a number of studies conducted in this direction. The authors in \cite{beh} have studied the characterization of linear Lie centralizable mappings at  zero on non-unital triangular algebras. In \cite{ghahjing}, linear Lie centralizers at the zero products on some operator algebras are studied that $\mathcal{B}(\mathcal{X})$ (the algebra of all bounded linear operators on $\mathcal{X}$) is a class of these operator algebras. In \cite{fad3}, linear Lie centralizable mappings at zero on a 2-torsion free unital generalized matrix algebra under some mild conditions, are described. As a result of \cite{fad3} (or \cite{ghahjing}), it is proved that if $\mathcal{X}$ is a Banach space over the complex field $\mathbb{C}$ and $\phi:\mathcal{B}(\mathcal{X}) \rightarrow \mathcal{B}(\mathcal{X})$ is a linear mapping satisfying $(\textbf{L}_{0})$, then $\phi$ is in the form $ \phi (A) = \lambda A + \mu (A) $ for all $A\in \mathcal{B}(\mathcal{X})$, where $ \lambda \in \mathbb{C} $ and $\mu: \mathcal{B}(\mathcal{X})\rightarrow \mathbb{C}I$ is a linear map such that $ \mu ([A,B]) = 0$ for all $A,B \in \mathcal{B}(\mathcal{X})$ with $AB=0$. In \cite{fad2}, the problem of characterizing linear Lie centralizable mappings at a non-trivial idempotent element on triangular algebras is considered, and in \cite{goo}, additive Lie centralizable mappings at a non-trivial idempotent element on a 2-torsion free unital prime ring are determined. Also in \cite{ghahmalay}, additive Lie centralizable mappings at a non-trivial idempotent element on a unital Banach algebra have been studied. As a corollary of \cite{ghahmalay}, we have that if $\mathcal{X}$ is a complex Banach space, $P\in \mathcal{B}(\mathcal{X})$ is a non-trivial idempotent (i.e., $P^{2}=P, P\neq 0, I$) and the additive mapping $\phi:\mathcal{B}(\mathcal{X}) \rightarrow \mathcal{B}(\mathcal{X})$ satisfies $(\textbf{L}_{P})$, then $ \phi (A) = \lambda A + \mu (A) $ for all $A\in \mathcal{B}(\mathcal{X})$, where $ \lambda \in \mathbb{C} $ and $\mu: \mathcal{B}(\mathcal{X})\rightarrow \mathbb{C}I$ is an additive mapping such that $ \mu ([A,B]) = 0$ for all $A,B \in \mathcal{B}(\mathcal{X})$ with $AB=P$. In \cite{liu4}, the authors studied the linear Lie centralizable mappings on generalized matrix algebras at a specific point $w$ that includes non-trivial idempotents. To find more results in this regard, we refer to \cite{fos2, gh4, ghi0, ghi, liu3} and the references therein. In all the conducted studies, the Lie centralizable mappings have been characterized at the zero point or non-trivial idempotent element or the elements that are close to the non-trivial idempotents. Now, in continuation of these studies, in this article we want to determine the additive Lie centralizable mappings at points other than zero points and non-trivial idempotents on $\mathcal{B}(\mathcal{X})$, and also obtain a generalization of the previous results on $\mathcal{B}(\mathcal{X})$. More specifically, we will prove the following results.
\begin{thm}\label{image}
Let $\mathcal{B}(\mathcal{X})$ be the algebra of all bounded linear operators on a complex Banach space $\mathcal{X}$, and let $\phi:\mathcal{B}(\mathcal{X}) \rightarrow \mathcal{B}(\mathcal{X})$ be an additive mapping. Then the following are equivalent:
\begin{itemize}
\item[(i)] $\phi$ satisfies $(\textbf{L}_{W})$, where $W\in \mathcal{B}(\mathcal{X}) $ and $\overline{W(\mathcal{X})}\neq \mathcal{X}$;
\item[(ii)] $ \phi (A) = \lambda A + \mu (A) $ for all $A\in \mathcal{B}(\mathcal{X})$, where $ \lambda \in \mathbb{C} $ and $\mu: \mathcal{B}(\mathcal{X})\rightarrow \mathbb{C}I$ is an additive mapping such that $ \mu ([A,B]) = 0$ for all $A,B \in \mathcal{B}(\mathcal{X})$ with $AB=W$.
\end{itemize}
\end{thm}
The proof of this theorem is provided in Section 2. A consequence of this theorem is the case when $W$ is zero or non-trivial idempotent (Corollaries \ref{zero} and \ref{idempotent}). This result is also presented in the third section, which shows that Theorem \ref{image} is also a generalization of the previous results.
\par 
If we can prove Theorem \ref{image} in the case where $\overline{W(\mathcal{X})}= \mathcal{X}$, then we actually obtain a characterization of the additive mappings on $\mathcal{B}(\mathcal{X})$ satisfying in $(\textbf{L}_{W})$ for every point $W\in \mathcal{B}(\mathcal{X})$. We were only able to prove this in the special case of Theorem \ref{invertible}, and the general case remains an open problem.
\begin{thm}\label{invertible}
Let $\mathcal{B}(\mathcal{X})$ be the algebra of all bounded linear operators on a complex Banach space $\mathcal{X}$, and let $\phi:\mathcal{B}(\mathcal{X}) \rightarrow \mathcal{B}(\mathcal{X})$ be an additive mapping. Then the following are equivalent:
\begin{itemize}
\item[(i)] $\phi$ satisfies $(\textbf{L}_{\xi I})$, where $\xi\in \mathbb{C}$ is fixed and $I$ is the identity operator;
\item[(ii)] $ \phi (A) = \lambda A + \mu (A) $ for all $A\in \mathcal{B}(\mathcal{X})$, where $ \lambda \in \mathbb{C} $ and $\mu: \mathcal{B}(\mathcal{X})\rightarrow \mathbb{C}I$ is an additive mapping such that $ \mu ([A,B]) = 0$ for all $A,B \in \mathcal{B}(\mathcal{X})$ with $AB=\xi I$.
\end{itemize}
\end{thm}
The proof of this theorem is given in Section 3. To prove this theorem, we first prove Lemma \ref{inverse} in which we characterize additive mappings on a unital Banach algebra $\U$ satisfying $\Ll$ for an invertible element $w\in \mathcal{Z}(\U)$, where $\mathcal{Z}(\U)$ is the center of $\U$. This characterization is in terms of additive commuting maps, which is an interesting result in itself. 

\section{Proof of Theorem \ref{image}}
In this section, we provide the proof of Theorem \ref{image} and some of its results. \\ \\
\textit{Proof of Theorem \ref{image}.}
(i)$\Rightarrow$(ii): If $dim\X=1$ then $\mathcal{B}(\X)=\mathbb{C}$ is commutative and it is enough to set $\mu:=\phi$ and $\lambda:=0$. Therefore the result is valid in this case. Now we assume $dim\X\geq 2$. Since  $ \overline{W(\X)}\neq \X $, there exists a nonzero element  $ x_{0}\in \X $ such that 
$ x_{0}\notin\overline{W(\X)} $. Suppose $ f_{0} \in X^{\ast} $ such that, $f_{0}(x_{0})=1$ 
and $ f_{0}(\overline{W(\X)})=0$. We set
$ P_{1}=x_{0}\otimes f_{0} $ and $ P_{2}=I-P_{1}. $ Then, $ P_{1} $ and $ P_{2} $ are nontrivial idempotents in $ \mathcal{B}(\X) $ such that $ P_{1}W=0 $ and $ P_{2}W=W $.\\
Consider the pierce decomposition of $ \mathcal{B}(\X) $ associated with the $ P_{1} $ as follows
\[ \mathcal{B}(\X)=\U_{11}+\U_{12}+\U_{21}+\U_{22} \]
in which $ \U_{ij}=P_{i}\mathcal{B}(\X)P_{j},(1\leqslant i,j\leqslant 2 ) $. 
Then $ \U_{11} $ and $ \U_{22} $ are unital Banach algebras with identities $ P_{1} $ and $ P_{2}$, respectively and $ \U_{12} $ is a unital Banach $ (\U_{11}, \U_{22})$-bimodule and $ \U_{21} $ is a unital Banach $ (\U_{22}, \U_{11})$-bimodule. We denote each element $ \U_{ij} $ by $ A_{ij}. $ It is clear that $ P_{1}=x_{0}\otimes f_{0} $ is a one-dimensional idempotent. Also for each $ A\in \mathcal{B}(\X) $ we have
\[ P_{1}AP_{1}=(x_{0}\otimes f_{0})A(x_{0}\otimes f_{0})=\lambda x_{0}\otimes f_{0} \]
for some $ \lambda \in \mathbb{C}. $ Thus $ \U_{11}=\mathbb{C}P_{1} $.
\par 
We complete the proof by proving the following steps.
\begin{cl}\label{l21}
For any $ A_{11} \in\U_{11} $ and $A_{22} \in\U_{22}$,
\[A_{22}\phi (A_{11})=\phi (A_{11}) A_{22} \]
and
\[ P_1 \phi (A_{11}) P_2 = P_2 \phi (A_{11} ) P_1 = 0. \]
\end{cl}
\begin{proof}
For any invertible element $ A_{22}\in \U_{22}$ and $ 0 \neq \lambda \in \mathbb{C} $, since
$ A_{22}(A_{22}^{-1}W)=W $ and $A_{22} ( A_{22}^{-1}W + \lambda P_1 ) = W$ ($A_{22}^{-1}$ denotes the inverse of $A_{22}$ in $\U_{22}$), we have
\begin{equation}\label{eq1}
\phi ( [ A_{22} , A_{22}^{-1}W ] ) = [A_{22} , \phi (A_{22}^{-1}W) ]
\end{equation}
and 
\[ \phi ( [ A_{22} ,  A_{22}^{-1}W + \lambda P_1 ] ) = [ A_{22} , \phi (  A_{22}^{-1}W + \lambda P_1 ) ]. \]
Consequently,
\begin{equation}\label{eq2}
\phi ( [ A_{22} , A_{22}^{-1}W] ) = [ A_{22} , \phi (A_{22}^{-1}W) + \phi ( \lambda P_1 ) ] .
\end{equation}
Comparing \eqref{eq1} and \eqref{eq2}, we see that $ [ A_{22}, \phi ( \lambda P_1 ) ] = 0  $ for any invertible element $A_{22} \in\U_{22}$. By the fact that each element of a Banach algebra is a sum of two invertible elements we get $ [ A_{22}, \phi ( \lambda P_1 ) ] = 0  $ for all $A_{22} \in\U_{22}$. Hence 
$ P_2 \phi ( \lambda P_1 ) = \phi ( \lambda P_1 ) P_2 $. So
\begin{equation*}\label{eq3}
P_2 \phi ( \lambda P_1 ) P_1 = P_1 \phi ( \lambda P_1 ) P_2 = 0 .
\end{equation*}
Now, for any $ A_{11} \in \U_{11} $, since $ A_{11} = \lambda P_1 $ for some $ \lambda \in \mathbb{C} $, we get desired results.
\end{proof}
\begin{cl}\label{l22}
For any $ A_{11} \in\U_{11} $ and $A_{22} \in\U_{22}$,
\[A_{11}\phi (A_{22})=\phi (A_{22}) A_{11} \]
and
\[ P_1 \phi (A_{22}) P_2 = P_2 \phi (A_{22} ) P_1 = 0 . \]
\end{cl}
\begin{proof}
For any invertible element $ A_{22}\in \U_{22}$, since
$ A_{22}( P_{1}+A_{22}^{-1}W )=W $
and
$ A_{22}(A_{22}^{-1}W)=W $, we have
\begin{equation}\label{G5}
\phi(\left[ A_{22},P_{1}+A_{22}^{-1}W\right])=\left[\phi(A_{22}), P_{1}+A_{22}^{-1}W\right] 
\end{equation}
and
\begin{equation}\label{G6}
\phi(\left[A_{22}, A_{22}^{-1}W\right])=\left[\phi(A_{22}), A_{22}^{-1}W \right]. 
\end{equation}
From \eqref{G5} and \eqref{G6}, we see that
\begin{equation}\label{G7} 
0=\phi(\left[A_{22}, P_{1}\right])=\left[\phi(A_{22}), P_{1}\right].  
\end{equation}
For each invertible element $ A_{22} $ in $\U_{22} $, from \eqref{G7}, It is obtained that
\[  P_{2}\phi(A_{22})P_{1}= P_{1}\phi(A_{22})P_{2}=0. \]
Now, for an arbitrary element $A_{22}$ of $\U_{22}$, we consider the non-negative integer $n$ as $n>\Vert A_{22} \Vert$. Since $nP_{2}-A_{22}$ is invertible in the Banach algebra $ \U_{22} $, it follows that 
\[ [\phi(A_{22}), P_{1}]=0\]
and 
\[ P_{2}\phi(A_{22})P_{1}= P_{1}\phi(A_{22})P_{2}=0 , \]
for any $ A_{22} \in \U_{22}$. Thus $[\phi(A_{22}), \lambda P_{1}]=0$ for every $\lambda\in \mathbb{C}$, and hence 
\[A_{11}\phi (A_{22})=\phi (A_{22}) A_{11} \]
for any $ A_{11} \in\U_{11} $ and $A_{22} \in\U_{22}$.
\end{proof}
\begin{cl}\label{l23}
For any $ A_{12} \in \U_{12} $,
\[ P_1 \phi (A_{12}) P_1 = P_2 \phi (A_{12} ) P_1 = P_2 \phi (A_{12} ) P_2 = 0 . \]
\end{cl}
\begin{proof}
For any $ A_{11}\in \U_{11} $ and $ A_{12}\in \U_{12}$, since
\[  P_{2}(W+A_{11}+A_{12})=W , \quad and \quad  P_{2}W=W \]
we have
\begin{equation}\label{G10}
\phi(\left[P_{2}, W+A_{11}+A_{12} \right])=\left[\phi(P_{2}), W+A_{11}+A_{12} \right] 
\end{equation}
and
\begin{equation} \label{G11}
 \phi(\left[P_{2}, W \right]=\left[\phi(P_{2}), W \right].  
\end{equation}
Hence, we can infer from  \eqref{G10} and \eqref{G11}, that
\begin{equation}\label{G12}
- \phi(A_{12})=[ \phi(P_{2}) , A_{11} +  A_{12} ]
\end{equation}
for all $ A_{11}\in \mathcal{A}_{11} $ and $ A_{12}\in \mathcal{A}_{12} $. We substitute $ A_{12}=0 $ into \eqref{G12}, thus $ [ \phi(P_{2}), A_{11} ]=0 $ which implies
\[  \phi(A_{12})=A_{12} \phi(P_{2}) - \phi(P_{2}) A_{12} . \]
By using Claim \ref{l22}, we see that
\begin{equation}\label{G12b}
 \phi(A_{12}) =A_{12}\phi(P_{2})P_2-P_1\phi(P_{2}) A_{12} . 
\end{equation}
Multiplying \eqref{G12b} once from left and right to $ P_1$, once from left and right to $ P_2$, and once from left to $P_2$ and from right to $P_1$, we conclude that 
\[ P_1 \phi (A_{12}) P_1 = P_2 \phi (A_{12} ) P_1 = P_2 \phi (A_{12} ) P_2 = 0 . \]
\end{proof}
\begin{cl}\label{l24}
For any $ A_{21} \in \U_{21} $,
\[ P_1 \phi (A_{21}) P_2  = 0 . \]
\end{cl}
\begin{proof}
For any $ A_{21}\in \U_{21} $, since
$ (A_{21}+P_{2})(W+P_{1}-A_{21})=W $ and $ (A_{21}+P_{2}) W=W$, we have
\[ \phi([A_{21}+P_{2}, W+P_{1}-A_{21} ])=[\phi(A_{21}+P_{2}), W+P_{1}-A_{21}] \]
and
\[
\phi([A_{21}+P_{2}, W ]=[\phi(A_{21}+P_{2}), W ]. \]
Now from previous equations and Claim \ref{l22}, it follows that 
\begin{equation*}
\begin{split}
0&=\phi( [A_{21}+P_{2}, P_{1}-A_{21} ]) \\
&=\left[\phi(A_{21}+P_{2}), P_{1} -  A_{21}\right] \\
&=\phi(A_{21})P_{1}-\phi(A_{21})A_{21}-\phi(P_{2})A_{21} \\
& \quad \quad \quad \quad  -P_{1}\phi(A_{21}) +A_{21}\phi(A_{21})+A_{21}\phi(P_{2})   \\
&= \phi(A_{21})P_{1}-\phi(A_{21})A_{21}-P_2 \phi(P_{2})A_{21}\\
& \quad \quad \quad \quad  -P_{1}\phi(A_{21})+A_{21}\phi(A_{21} )+A_{21}\phi(P_{2}) P_1 
\end{split}
\end{equation*}
for each $A_{21}$ in $\mathcal{A}_{21}$.
Now we multiply both sides of the above equation by $ P_{2} $ on the right and by $ P_{1} $ on the left. Therefore
\[ P_{1}\phi(A_{21})P_{2}=0  \]
for all $A_{21}$ in $\mathcal{A}_{21}$.
\end{proof}
\begin{cl}\label{l25}
There exsists an additive mapping $ h_{1}:\U_{11}\rightarrow \mathbb{C}I $ such that 
\[\phi(A_{11})-h_{1}(A_{11}) \in \U_{11}\] 
for any $ A_{11}\in \U_{11} $.
\end{cl}
\begin{proof}
From Claim \ref{l21}, it follows that $P_2 \phi(A_{11}) A_{22}=A_{22} \phi(A_{11}) P_2$ for all $ A_{11} \in\U_{11} $ and $A_{22} \in\U_{22}$. So $P_2\phi(A_{11}) P_2\in \mathcal{Z}(\U_{22})$. Since $\U_{22}\cong \mathcal{B}(kerf_0)$ (algebraic isomorphism) and $P_2$ is the unity of $\U_{22}$, it is obtained that $\mathcal{Z}(\U_{22})=\mathbb{C}P_2$. Now, according to these results, for each $A_{11} \in\U_{11}$, there exists an $\lambda_{A_{11}}\in \mathbb{C}$ such that $P_2\phi(A_{11}) P_2=\lambda_{A_{11}}P_2$. Define the well-defined mapping $ h_{1}:\U_{11}\rightarrow \mathbb{C}I$ by $h_1(A_{11})=\lambda_{A_{11}}I$. Given the fact that $\phi$ is additive, it follows that $h_1$ is additive. From Claim \ref{l21} and the definition of $h_1$, we have
\begin{equation*}
\begin{split}
\phi(A_{11})-h_{1}(A_{11})&=P_1 \phi(A_{11})P_1 +P_2\phi(A_{11})P_2\\
& \quad \quad \quad \quad \quad -P_1 h_1(A_{11}) P_1 -P_2 h_1(A_{11}) P_2\\
&= P_1 \phi(A_{11})P_1 +\lambda_{A_{11}}P_2-\lambda_{A_{11}}P_1 -\lambda_{A_{11}}P_2\\
&=P_1 \phi(A_{11})P_1-\lambda_{A_{11}}P_1 \in \U_{11}
\end{split}
\end{equation*}
for all $ A_{11}\in \U_{11} $.
\end{proof}
\begin{cl}\label{l26}
There exsists an additive mapping $h_{2}:\U_{22}\rightarrow \mathbb{C}I $ such that
\[ \phi(A_{22})-h_{2}(A_{22}) \in \U_{22} \]
for any $ A_{22}\in \U_{22}$.
\end{cl}
\begin{proof}
Since $\U_{11}=\mathbb{C}P_1$, for every $A_{22}\in \U_{22}$ there exists an $\lambda_{A_{22}}\in \mathbb{C}$ such that $P_1\phi(A_{22}) P_1=\lambda_{A_{22}}P_1$. Define the well-defined mapping $ h_{2}:\U_{22}\rightarrow \mathbb{C}I$ by $h_2(A_{22})=\lambda_{A_{22}}I$. Considering that $\phi$ is additive, it follows that $h_2$ is additive. From Claim \ref{l22} and the definition of $h_2$, we have
\begin{equation*}
\begin{split}
\phi(A_{22})-h_{2}(A_{22})&=P_1 \phi(A_{22})P_1 +P_2\phi(A_{22})P_2- \lambda_{A_{11}}I\\
&=\lambda_{A_{22}}P_1+ P_2 \phi(A_{22})P_2 -\lambda_{A_{22}}P_1 -\lambda_{A_{22}}P_2\\
&=P_2 \phi(A_{22})P_2-\lambda_{A_{22}}P_2 \in \U_{22}
\end{split}
\end{equation*}
for all $ A_{22}\in \U_{22} $.
\end{proof}
\begin{cl}\label{l27}
There exsists an additive mapping $ h_{3}:\U_{21}\rightarrow \mathbb{C}I $ such that
\[ \phi(A_{21})-h_{3}(A_{21}) \in \U_{21}\]
for any $ A_{21}\in \U_{21} $.
\end{cl}
\begin{proof}
For any $ A_{12}\in \U_{12} $ and $ A_{21} \in \U_{21} $, since
\[ (P_{2}+A_{21})(W-A_{21} A_{12}+A_{12})=W \]
and 
 \[ (P_{2}+A_{21})W=W ,\] 
we have
\[ \phi(\left[P_{2}+A_{21}, W-A_{21} A_{12}+A_{12} \right])=\left[\phi(P_{2}+A_{21}), W-A_{21} A_{12}+A_{12} \right] \]
and
\[ \phi(\left[P_{2}+A_{21}, W\right])=\left[\phi(P_{2}+A_{21}), W \right]. \]
Comparing the above two equations, we see that 
\[\phi(\left[P_{2}+A_{21}, -A_{21}A_{12}+A_{12} \right]) =\left[\phi(P_{2}+A_{21}), -A_{21}A_{12}+A_{12} \right]. \]
Thus
\begin{equation*}
\begin{split}
&\phi(A_{21} A_{12} ) + \phi( A_{21} A_{12}A_{21} )- \phi( A_{12} ) - \phi( A_{12}A_{21})=\\
&\quad-\phi(P_{2})A_{21} A_{12}+\phi(P_{2})A_{12}-\phi(A_{21})A_{21}A_{12}+\phi(A_{21}) A_{12} \\
  & \quad +A_{21}A_{12}\phi(P_{2})+A_{21}A_{12}\phi(A_{21})- A_{12}\phi(P_{2})-A_{12}\phi(A_{21}). 
\end{split}
\end{equation*}
Multiplying this equation on the right by $ P_{2} $ and on the left by $ P_{1}$, and using Claims \ref{l22} and \ref{l24}, we obtain
\begin{equation}\label{G16}
-\phi(A_{12}) =P_{1}\phi(P_{2}) A_{12}+P_{1}\phi(A_{21})A_{12} -A_{12}\phi(P_{2})P_{2}-A_{12}\phi(A_{21})P_{2}.
\end{equation} 
By setting $ A_{21}=0 $ in \eqref{G16}, we get
\[ -\phi(A_{12})=P_{1}\phi(P_{2})A_{12}-A_{12}\phi(P_{2})P_{2}. \]
So from \eqref{G16}, we see that 
\begin{equation}\label{G16b}
P_{1}\phi(A_{21})A_{12}=A_{12}\phi(A_{21}) P_{2} . 
\end{equation} 
Since $\U_{11}=\mathbb{C}P_1$, for every $A_{21}\in \U_{21}$ there exists a $\lambda_{A_{21}}\in \mathbb{C}$ such that $P_1\phi(A_{21}) P_1=\lambda_{A_{21}}P_1$. According to \eqref{G16b}, we arrive at 
\[\lambda_{A_{21}}A_{12}=A_{12}\phi(A_{21}) P_{2},  \]
and hence 
\[ A_{12}(\lambda_{A_{21}}P_2 -P_2 \phi(A_{21}) P_{2})=0\]
for all $A_{12}\in \U_{12} $. Since $\mathcal{B}(\X)$ is a prime algebra, it follows that 
\[ P_2\phi(A_{21}) P_2=\lambda_{A_{21}}P_2 \]
for all $A_{21}\in \U_{21} $. Define the well-defined mapping $ h_{3}:\U_{21}\rightarrow \mathbb{C}I$ by $h_3(A_{21})=\lambda_{A_{21}}I$. Given the fact that $\phi$ is additive, it follows that $h_3$ is additive. From Claim \ref{l24} and the definition of $h_3$, we have
\begin{equation*}
\begin{split}
\phi(A_{21})-h_{3}(A_{21})&=P_1 \phi(A_{21})P_1 +P_2\phi(A_{21})P_1+P_2 \phi(A_{21})P_2- \lambda_{A_{21}}I\\
&=\lambda_{A_{21}}P_1+ P_2 \phi(A_{21})P_1+ \lambda_{A_{21}}P_2-\lambda_{A_{21}}I \\
&=P_2 \phi(A_{21})P_1\in \U_{21}
\end{split}
\end{equation*}
for all $ A_{21}\in \U_{21} $.
\end{proof}
Now, for any $ A = A_{11} + A_{12} + A_{21} + A_{22} \in \mathcal{B}(\X)$, we define two additive mappings $ \mu : \mathcal{B}(\X) \to \mathbb{C}I $ and $ \psi : \mathcal{B}(\X) \to \mathcal{B}(\X)$ by
\[ \mu ( A ) = h_1 ( A_{11} ) +h_3 ( A_{21} ) +  h_2 ( A_{22} )  ~~~~ \text{and} ~~~~ \psi (A) = \phi (A) - \mu(A) . \]
By Claims \ref{l21}-\ref{l27}, it is clear that $ \psi (\U_{ij} ) \subseteq \U_{ij} $ for $ 1 \leq i \neq j \leq 2$, and $ \psi (A_{12} ) = \phi ( A_{12} ) $ for any $ A_{12}\in \U_{12} $. 
\begin{cl}\label{l28}
$ \psi $ satisfies $\psi(A)=\lambda A $ for all $A\in \mathcal{B}(\X)$, where $\lambda\in \mathbb{C}$.
\end{cl}
\begin{proof}
First, we show that $\psi(I)=\lambda I$ for some $\lambda\in \mathbb{C}$. For every $ A_{21}\in \U_{21} $, since
\[ (P_2-A_{21})(W+A_{21}+P_{1})=W \]
and
\[ (P_2-A_{21})W=W, \]
it follows that
\begin{equation*}
\begin{split}
\phi(\left[P_2-A_{21}, W+A_{21}+P_{1} \right])&=\left[\phi(P_2-A_{21}), W+A_{21}+P_{1} \right] \\&
=\left[P_2-A_{21}, \phi (W+A_{21}+P_{1}) \right]
\end{split}
\end{equation*}  
and
\[ \phi([P_2-A_{21},W])=[\phi(P_2-A_{21}),W ]=[P_2-A_{21},\phi(W) ]. \]
From these equations and the fact that $ \psi (\U_{ij} ) \subseteq \U_{ij} $ for $ 1 \leq i \neq j \leq 2$, we have
\begin{equation*}
\begin{split}
0&=\phi(\left[P_2-A_{21}, A_{21}+P_{1} \right] ) \\ 
&=\left[\phi(P_2-A_{21}),A_{21}+P_{1} \right]\\ 
&=\left[\psi(P_2-A_{21}), A_{21}+P_{1} \right]  \\ 
&=\psi(P_2)A_{21}-\psi(A_{21}) 
\end{split}
\end{equation*}
and 
\begin{equation*}
\begin{split}
0&=\phi(\left[P_2-A_{21}, A_{21}+P_{1} \right] ) \\ 
&=\left[P_2-A_{21}, \phi(A_{21}+P_{1}) \right]  \\ 
&=\left[P_2-A_{21}), \psi(A_{21}+P_{1}) \right] \\
&=\psi(A_{21})-A_{21}\psi(P_{1}) 
\end{split}
\end{equation*}
Thus
\begin{equation}\label{G16bb}
\psi(A_{21})=\psi(P_2)A_{21}=A_{21}\psi(P_{1}) 
\end{equation}  
for every $ A_{21}\in \U_{21} $. Since $\U_{11}=\mathbb{C}P_1$ and $\psi(P_{1})\in \U_{11}$, there exists a $\lambda \in \mathbb{C}$ such that $\psi(P_{1})=\lambda P_1$. According to \eqref{G16bb}, we get
\[\psi(P_2)A_{21}=\lambda A_{21} ,  \]
and hence 
\[ (\psi(P_2)-\lambda P_2)A_{21}=0\]
for all $A_{21}\in \U_{21} $. Since $\mathcal{B}(\X)$ is a prime algebra, it follows that 
\[\psi(P_2)=\lambda P_2.  \]
Consequentially,
\[\psi(I)=\psi(P_1)+\psi(P_2)=\lambda P_1+\lambda P_2=\lambda I. \]
From \eqref{G16bb} and $\psi(P_2)\in \U_{22}$, we see that 
\[\psi(A_{21})=A_{21}\psi(P_{1}) =A_{21}\psi(P_{1}) +A_{21}\psi(P_{2})=A_{21}\psi(I)=\lambda A_{21}  \]
for every  $ A_{21}\in \U_{21} $. Given that $ \psi (A_{12} ) = \phi ( A_{12} ) $ for any $ A_{12}\in \U_{12} $, it follows from \eqref{G12b} that
\[\psi(A_{12}) =A_{12}\phi(P_{2})P_2-P_1\phi(P_{2}) A_{12}.  \]
From the definition of $h_2$ we have $P_1\phi(P_{2})P_1=h_2(P_2)P_1$ and therefore from the definitions of $\psi$ and $\mu$ we have
\begin{equation*}
\begin{split}
\psi(A_{12}) &=A_{12}\phi(P_{2})P_2-h_2(P_2) A_{12}\\&
=A_{12}(\phi(P_2)-\mu(P_2))P_2\\&
=A_{12}\psi(P_2)
\end{split}
\end{equation*}
for every  $ A_{12}\in \U_{12} $. In view of the fact that $\psi(P_1)\in \U_{11}$, we obtain 
\[\psi(A_{12})=A_{12}\psi(P_{2}) =A_{12}\psi(P_{2}) +A_{12}\psi(P_{1})=A_{12}\psi(I)=\lambda A_{12}  \]
for every $ A_{12}\in \U_{12} $. For every invertible element $ A_{22}\in \U_{22} $ and $ B_{21}\in \U_{21} $ we have
\[ (A_{22}-A_{22}B_{21})(A_{22}^{-1}W+B_{21}+P_{1})=W \]
and
\[ (A_{22}-A_{22}B_{21})(A_{22}^{-1}W)=W. \]
By applying $\phi$ to these equations, we arrive at
\[ \phi(\left[A_{22}-A_{22}B_{21}, A_{22}^{-1}W+B_{21}+P_{1} \right])=\left[\phi(A_{22}-A_{22}B_{21}), A_{22}^{-1}W+B_{21}+P_{1} \right]  \]
and
\[ \phi([A_{22}-A_{22}B_{21},A_{22}^{-1}W])=[\phi(A_{22}-A_{22}B_{21}),A_{22}^{-1}W]. \]
Now, from these equations it follows that 
\[0=\phi(\left[A_{22}-A_{22}B_{21}, B_{21}+P_{1} \right] ) =\left[\phi(A_{22}-A_{22}B_{21}),B_{21}+P_{1} \right] \] 
and therefore 
\[ \psi(A_{22}B_{21})=\psi(A_{22})B_{21} \]
for every invertible element $ A_{22}\in \U_{22} $ and any $ B_{21}\in \U_{21} $. By the fact that each element of a Banach algebra is a sum of two invertible elements we get 
\[ \psi(A_{22}B_{21})=\psi(A_{22})B_{21} \]
for every $ A_{22}\in \U_{22} $ and  $ B_{21}\in \U_{21} $. Consequently 
\[\psi(A_{22})B_{21}=\psi(A_{22}B_{21})=\lambda A_{22}B_{21} \]
for every $ A_{22}\in \U_{22} $ and  $ B_{21}\in \U_{21} $. Since $\mathcal{B}(\X)$ is a prime algebra, it follows that 
\[\psi(A_{22})= \lambda A_{22}\]
for every $ A_{22}\in \U_{22} $. For every $ A_{11}\in \U_{11} $ and $ B_{21}\in \U_{21} $ we have
\[ (P_{2}+B_{21})(A_{11}-B_{21}A_{11}+W)=W \]
and
\[ (P_2 +B_{21})W=W. \]
By applying $\phi$ to these equations, we get
\[ \phi ([P_{2}+B_{21},A_{11}-B_{21}A_{11}+W])=[P_{2}+B_{21},\phi(A_{11}-B_{21}A_{11}+W)] \]
and
\[\phi([P_2 +B_{21},W])=[P_2 +B_{21}, \phi(W)]. \]
Comparing these equations, we see that
\[0=\phi ([P_{2}+B_{21},A_{11}-B_{21}A_{11}])=[P_{2}+B_{21},\phi(A_{11}-B_{21}A_{11})]  \]
and hence
\[\psi(B_{21}A_{11})=B_{21}\psi(A_{11}) \]
for every $ A_{11}\in \U_{11} $ and $ B_{21}\in \U_{21} $. So
\[B_{21}\psi(A_{11})= \psi(B_{21}A_{11})=\lambda B_{21}A_{11}\]
for every $ A_{11}\in \U_{11} $ and $ B_{21}\in \U_{21} $. Thus
\[\psi(A_{11})=\lambda A_{11} \]
for every $ A_{11}\in \U_{11} $. According to the obtained results, for each $A=A_{11}+A_{12}+A_{21}+A_{22}\in \mathcal{B}(\X)$ we have
\[\psi(A)=\psi(A_{11})+\psi(A_{12})+\psi(A_{21})+\psi(A_{22})=\lambda A_{11}+\lambda A_{12}+\lambda A_{21}+\lambda A_{22}=\lambda A. \]
\end{proof}
We are now in a position to complete the proof of (i)$\Rightarrow$(ii). According to the definition of $\psi$ and $\mu$, we have
\[  \phi (A)=\psi (A)  +\mu(A)=\lambda A + \mu(A) \]
for all $A\in \mathcal{B}(\X)$, where $\lambda \in \mathbb{C}$ and $ \mu : \mathcal{B}(\X) \to \mathbb{C}I$ is an additive mapping. For any $A,B\in \mathcal{B}(\X)$ with $AB=W$ we have
\begin{equation*}
\begin{split}
\mu([A,B])&=\phi([A,B])-\lambda [A,B] \\&
=[\phi(A),B]-\lambda [A,B] \\&
=[\mu(A),B]+[\lambda A, B]-\lambda [A,B] =0.
\end{split}
\end{equation*}
\par 
(ii)$\Rightarrow$(i): It is clear. $\quad\quad \quad \quad \quad\quad \quad \quad\quad\quad \quad \quad \quad\quad \quad \quad \quad\quad \quad \quad\quad\quad \blacksquare $
\\ \par 
We have the following corollaries from Theorem \ref{image}, which were obtained previously, and therefore this Theorem \ref{image} is a generalization of these previous results.
\begin{cor}\label{zero}
Let $\mathcal{B}(\mathcal{X})$ be the algebra of all bounded linear operators on a complex Banach space $\mathcal{X}$, and let $\phi:\mathcal{B}(\mathcal{X}) \rightarrow \mathcal{B}(\mathcal{X})$ be an additive mapping. Then the following are equivalent:
\begin{itemize}
\item[(i)] $\phi$ satisfies $(\textbf{L}_{0})$;
\item[(ii)] $ \phi (A) = \lambda A + \mu (A) $ for all $A\in \mathcal{B}(\mathcal{X})$, where $ \lambda \in \mathbb{C} $ and $\mu: \mathcal{B}(\mathcal{X})\rightarrow \mathbb{C}I$ is an additive mapping such that $ \mu ([A,B]) = 0$ for all $A,B \in \mathcal{B}(\mathcal{X})$ with $AB=0$.
\end{itemize}
\end{cor}
\begin{proof}
Let $W=0$. We have $\overline{W(\mathcal{X})}=0\neq \mathcal{X}$ and therefore, by Theorem \ref{image}, the result is clear.
\end{proof}
We recall that an idempotent $P\in \mathcal{B}(\mathcal{X})$ is non-trivial whenever $P\neq 0$ and $P\neq I$.
\begin{cor}\label{idempotent}
Let $\mathcal{B}(\mathcal{X})$ be the algebra of all bounded linear operators on a complex Banach space $\mathcal{X}$, and let $\phi:\mathcal{B}(\mathcal{X}) \rightarrow \mathcal{B}(\mathcal{X})$ be an additive mapping. Then the following are equivalent:
\begin{itemize}
\item[(i)] $\phi$ satisfies $(\textbf{L}_{P})$, where $P\in \mathcal{B}(\mathcal{X})$ is a non-trivial idempotent;
\item[(ii)] $ \phi (A) = \lambda A + \mu (A) $ for all $A\in \mathcal{B}(\mathcal{X})$, where $ \lambda \in \mathbb{C} $ and $\mu: \mathcal{B}(\mathcal{X})\rightarrow \mathbb{C}I$ is an additive mapping such that $ \mu ([A,B]) = 0$ for all $A,B \in \mathcal{B}(\mathcal{X})$ with $AB=P$.
\end{itemize}
\end{cor}
\begin{proof}
Since $P$ is a non-trivial idempotent, $P(\X)$ is a closed subspace of $\X$ and $P(\X)\neq \X$ (see \cite[Theorem 13.2]{conway}). So $\overline{P(\mathcal{X})}\neq \mathcal{X}$ and therefore, by Theorem \ref{image}, the result is clear.
\end{proof}

\section{Proof of Theorem \ref{invertible}}
To prove this theorem, we first need to introduce commuting maps. Let $\mathcal{U}$ be an algebra. Recall that a map $\theta$ of $\mathcal{U}$ into itself is called a \textit{commuting map} if $[\theta(a),a]=0$ for all $a \in \mathcal{U}$. One can easily check that each Lie centralizer is a commuting map, but the converse is, in general, not true (see \cite[Example 2.3]{ghahjing}).Commuting maps have played a crucial role in the development of the theory of functional identities. We refer the reader to the survey paper \cite{bre2} in which the author presented the development of the theory of commuting maps and their applications.
\par 
As we mentioned in the introduction, the following lemma is interesting in itself.
\begin{lem}\label{banach}
Let $\U$ be a complex Banach algebra with unity $1$, and $w$ is an invertible element in $\mathcal{Z}(\U)$. If $\phi:\U \rightarrow \U$ is an additive mapping satisfying $\Ll$, then $\phi$ is a commuting map.
\end{lem}
\begin{proof}
For each invertible element $a\in \U$ we have $aa^{-1}w=w$. So 
\[[\phi(a),a^{-1}w]=0, \]
and hence $\phi(a)a^{-1}w-a^{-1}w\phi(a)=0$. In view of the fact that $w\in \mathcal{Z}(\U)$ is invertible, we conclude that
\begin{equation}\label{inverse}
\phi(a)a=a\phi(a).
\end{equation} 
We consider the non-negative integer $n>\Vert a \Vert$ for the invertible element $a\in \U$. So $n1-a$ is invertible and from \eqref{inverse} it follows that $\phi(n1-a)(n1-a)=(n1-a)\phi(n1-a)$, and hence $\phi(n1)a=a\phi(n1)$. Since $\phi$ is additive we get $\phi(1)a=a\phi(1)$ for all invertible element $a\in \U$. By the fact that each element of a Banach algebra is a sum of two invertible elements we conclude that $\phi(1)\in \mathcal{Z}(\U)$. 
\par 
Now, for an arbitrary element $a$ of $\U$, we consider the non-negative integer $n$ as $n>\Vert a \Vert$. Since $n1-a$ is invertible, it follows from \eqref{inverse} that $\phi(n1-a)(n1-a)=(n1-a)\phi(n1-a)$. Given that $\phi(1)\in \mathcal{Z}(\U)$, we see that $\phi(a)a=a\phi(a)$. Therefore, $\phi$ is a commuting map.
\end{proof}
From the proof of the above lemma, it also follows that on Banach algebras additive commuting maps in invertible elements are commuting maps.
\\ \\
\textit{Proof of Theorem \ref{invertible}.}
(i)$\Rightarrow$(ii): Every central and invertible member in $\mathcal{B}(\mathcal{X})$ is of the form $\xi I$, where $\xi\in \mathbb{C}$ and $I$ is the identity operator. Now, considering this point, it follows from Lemma \ref{banach} that $\phi$ is an additive commuting map. $\mathcal{B}(\mathcal{X})$ is a prime algebra and the extended centroid of $\mathcal{B}(\mathcal{X})$ is equal to $\mathcal{Z}(\mathcal{B}(\mathcal{X}))=\mathbb{C}I$. From \cite[Theorem A]{bre1} it follows that $\phi(A)=\lambda A +\mu(A)$ ($A\in \mathcal{B}(\mathcal{X})$), where $\lambda$ is in the extended centroid of $\mathcal{B}(\mathcal{X})$ and $\mu$ is an additive map from $\mathcal{B}(\mathcal{X})$ into the its extended centroid. Hence $\lambda \in \mathbb{C}$ and $\mu(\mathcal{B}(\mathcal{X})) \subseteq \mathbb{C}I$. Now for any $A,B\in \U$ with $AB=\xi I$ we have
\begin{equation*}
\begin{split}
\mu([A,B])&=\phi([A,B])-\lambda [A,B] \\&
=[\phi(A),B]-\lambda [A,B] \\&
=[\mu(A),B]+[\lambda A, B]-\lambda [A,B] =0.
\end{split}
\end{equation*}
\par 
(ii)$\Rightarrow$(i): It is clear. $\quad\quad \quad \quad \quad\quad \quad \quad\quad\quad \quad \quad \quad\quad \quad \quad \quad\quad \quad \quad\quad\quad \blacksquare $

\subsection*{Declarations}
\begin{itemize}
\item[•] \textbf{Author's contribution:} All authors contributed to the study conception and design and approved the final manuscript. 
\item[•] \textbf{Funding:} The authors declare that no funds, grants, or other support were received during the preparation of this manuscript.
\item[•] \textbf{Conflict of interest:} On behalf of all authors, the corresponding author states that there is no conflict of interest. 
\item[•] \textbf{Data availability:} Data sharing not applicable to this article as no datasets were generated or analysed during the current study. 
\end{itemize}

\subsection*{Acknowledgment}
The authors like to express their sincere thanks to the referee(s) for this paper.




\end{document}